\documentclass[11pt]{article}
\usepackage{amsfonts}
\usepackage{amssymb,amsmath}
\usepackage{fancyhdr}
\usepackage{theorem}
\usepackage{times}
\theoremstyle{plain}
\newtheorem{theorem}{Theorem}[section]

\newtheorem{lemma}[theorem]{Lemma}

{\theorembodyfont{\normalfont}

\newenvironment{proof}[1][Proof]{\noindent\textbf{#1.} }{\ \rule{0.5em}{0.5em}}
\begin{document}
 \title{ On SDE associated with  continuous-state branching processes  conditioned to never be extinct}
\author{ M.C. Fittipaldi\thanks{DIM--CMM,  UMI 2807 UChile-CNRS, Universidad de Chile, Casilla 170-3, Correo 3, Santiago,
Chile ;  Supported by  Basal-CONICTY; mfittipaldi@dim.uchile.cl}  
\and J. Fontbona \thanks{DIM--CMM,  UMI 2807 UChile-CNRS, Universidad de Chile, Casilla 170-3, Correo 3, Santiago,
Chile ;  Partially supported by  Basal-CONICTY; fontbona@dim.uchile.cl. } }

\maketitle

\begin{abstract}

 We study the  pathwise description of  a (sub-)critical continuous-state branching process (CSBP)  conditioned to be never extinct, 
as  the solution to a stochastic differential equation driven by  Brownian motion  and Poisson point measures. 
The interest of our approach,  which relies on applying Girsanov theorem  on the SDE that describes  the unconditioned CSBP, 
is that it  points out an explicit mechanism  to build  the immigration term appearing in the conditioned process, by  randomly selecting   jumps  of the original one.  
These techniques should  also  be useful to represent more general $h$-transforms of diffusion-jump processes. 
\end{abstract}

\begin{center}
                    Non-extinction; Immigration. 

\end{center}

 \section{Introduction and preliminaries}

  Stochastic differential equations (SDE)  representing continuous-state branching processes (CSBP) or CSBP with immigration (CBI) 
have attracted  increasing attention  in the last years, as  powerful tools for studying  pathwise and distributional properties of 
these processes as well as some scaling limits, see e.g.  Dawson and Li \cite{DL0},  \cite{DL} , Lambert \cite{L2}, Fu and Li \cite{FL} and Caballero {\it et al.} \cite{CLU}.
 
 In this note, we are interested in SDE representations  for (sub)-critical CSBP conditioned to never be extinct.
It is well known  that such conditioned CSBP   correspond to CBIs with particular immigration mechanisms (see \cite{RR}).  Thus, it is  possible to obtain  SDE representations  for them by using  general  results and techniques developed in  some of the aforementioned works,  see \cite{DL0} and \cite{FL}.  
However, our goal  is to directly obtain  such representation by rather 
using the  fact that  the law of the conditioned CSBP is obtained from the one  of the non conditioned process,  by means of an explicit $h-$transform.
 In accordance with that relation between the laws and to the ``spine'' or immortal particle picture of the conditioned process (\cite{RR},  \cite{EV}), one should expect to 
identify, after  measure change, copies of the original driving random processes and an independent subordinator  accounting for immigration.
   Our proof  will show how to obtain these processes by using Girsanov theorem and an enlargement of the probability space in order to 
select  by a suitable marking procedure those  jumps of the original (non conditioned) process that will constitute (or will not)  the immigrants.
  The enlargement of the probability space and the marking  procedure  are both inspired in a  construction of Lambert \cite{L2} on 
stable L\'evy processes. They are also reminiscent of the sized biased tree representation of measure changes for Galton-Watson trees  
(Lyons {\it et. al} \cite{LPR}) or for  branching Brownian motions (see e.g. Kyprianou  \cite{K} and Engl\"anger and Kyprianou  \cite{EK}),
 but we do not aim at fully developing those  ideas in the present  framework. In a related  direction,    in a recently posted article \cite{H} 
H\'enard  obtains  the same SDE description of the conditioned CSBP, using the look-down particle representation of CSBP  of Donnelly and Kurtz \cite{DK}.

We start by recalling  some definitions and classic results about CSBPs and L\'evy processes along the lines of 
\cite[Chap. 1,2 and 10]{Ky}, in particular the relationship between them through the Lamperti transform. 
(We also refer the reader to  Le Gall \cite{LG}  and Li \cite{Li} for further background on CSBP). 

\subsection{Continuous-state branching processes}

 Continuous-state branching processes (CSBP) were introduced by Jirina \cite{J} in 1958. Later, Lamperti \cite{L1} 
showed that they can be obtained as scaling limits of a sequence of Galton-Watson processes. 
A CSBP with probability laws given the initial state 
$\{ \mathbb{P}_{x} : x\geq0\}$ is a c\`{a}dl\`{a}g $[0,\infty)$-valued strong Markov processes $Z= \{Z_{t}: t\geq 0\}$ satisfying the
branching property. That is,  for any $t\geq0$ and $z_{1},z_{2} \in [0,\infty)$, $Z_{t}$  under $\mathbb{P}_{z_{1}+z_{2}}$ has the same   law as the 
independent sum $Z_{t}^{(1)} + Z_{t}^{(2)}$, where the distribution of $Z_{t}^{(i)}$ is equal to that of $Z_{t}$ 
under $\mathbb{P}_{z_{i}}$ for $i=1,2$. Usually, $Z_{t}$ represents the population at time $t$ descending from an initial population $x$.   The law of $Z$ is completely characterized by its Laplace transform
\[
\mathbb{E}_{x}(e^{- \theta Z_{t}})\,=\,e^{-xu_{t}(\theta)}, \,\,\forall\, x>0,\, t\geq0, 
\]
where $u$ is a differentiable function in $t$  satisfying    
     \begin{equation}\label{eqdif}
     \left \{ 
     \begin{array}{ll}
      \displaystyle{\frac{\partial u_{t}}{ \partial t}(\theta)+ \psi(u_{t}(\theta))}=0\\                                                                        
      \\ u_{0}(\theta)=\theta,
     \end{array}
     \right.
     \end{equation}   
     and  $\psi$ is called the \textit{branching mechanism} of $Z$, which has the form 
     \begin{equation}\label{psi}
     \psi(\lambda)=-q-a\lambda+\frac{1}{2}\sigma^{2}\lambda^{2}
                   +\int_{(0,\infty)}(e^{-\lambda x} -1
                   +\lambda x \textbf{1}_{(x<1)})\Pi(dx)  \quad \lambda\geq 0, 
     \end{equation}
for some $q\geq0,\,a\in \mathbb{R},\, \sigma\geq0$ and $\Pi$ a measure supported in $(0,\infty)$ such that\\
$\int_{(0,\infty)} (1\wedge x^{2})\Pi(dx)\,<\,\infty.$
In particular, $\psi$ is  the characteristic exponent of a \textit{spectrally positive} L\'evy process, i.e. one 
with no negative jumps.    Since clearly,  $  \mathbb{E}_{x}(Z_{t})=xe^{-\psi'(0+)t}$, 
 defining $\rho:=\psi'(0+)$ one has the following classification of CSBPs :
     \begin{itemize}
       \item[(i)]
       \textnormal{subcritical}, if $\rho>0$,
       \item[(ii)]
       \textnormal{critical}, if $\rho=0$ and
       \item[(iii)]
       \textnormal{supercritical}, if $\rho<0$,
     \end{itemize}
according to whether the process will, on average, decrease, remain constant or increase. 

In the following, we will assume that $Z$  is \textit{conservative},  i.e.  $\forall$ $t>0$, $\mathbb{P}_{x}(Z_{t}<\infty)= 1$. 
By Grey (1974), this is true if and only if  $ \int_{0^{+}} \frac{d\xi}{|\psi(\xi)|}=\infty$, so it is sufficient to asume 
$\psi(0)=0$ and $|\psi'(0+)|<\infty$.

\subsection{L\'evy Processes and their connection with CSBP} 

\hspace*{0.5cm} Let $X=\{X_{t}: t \geq 0\}$ be a spectrally positive  L\'evy process with characteristic exponent $\psi$ given by 
\eqref{psi} with $q=0$, and initial state  $x\geq0$ . 
By the L\'evy-Ito  decomposition it is well known that it can be written as   the following sum of  independent processes 
 \begin{equation*}
       X_{t}=       x+at+\sigma B_{t}^{X} + \int_{0}^{t}\int_{1}^{\infty} rN^{X}(ds,dr)+ \int_{0}^{t}\int_0^1r\tilde{N}^{X}(ds,dr), 
    \end{equation*}
where $a$ is a real number, $\sigma \geq 0$, $B^{X}$ is a Brownian motion,
$N^{X}$ is an independent Poisson measure on $[0,\infty)\times(0,\infty)$ with intensity measure $dt\times\Pi(dr)$ and  $\tilde{N}^{X}(dt,dr):=N^{X}(dt,dr)-dt\Pi(dr)$ denotes the compensated measure associated to $N^{X}$ (the last integral  thus being a square integrable martingale of  compensated jumps of magnitude less than unity).

  Lamperti \cite{L} established a one-to-one correspondence between CSBPs and spectrally positive L\'evy processes   via a random time 
change. More precisely, for a L\'evy process $X$ as above the  process  
\[
Z:=\{Z_{t}=X_{\theta_{t}\wedge T_{0}}: t\geq 0 \}, 
\]
where $T_{0} = \inf \{t>0:X_{t}=0\}$ and $\theta_{t}=\inf \left\{s>0:\int_{0}^{s}\frac{du}{X_{u}}>t \right\}$,
is a continuous-state branching process with branching mechanism $\psi$ and initial value $Z_{0}=x$.
  Conversely, given $Z=\{Z_{t}:t\geq0\}$ a CSBP with branching mechanism $\psi$, such that
$Z_{0}=x>0$, we have that
\[
X:=\{X_{t}=Z_{\varphi_{t}\wedge T}: t\geq0\}, 
\]
where $T=\inf\{t>0: Z_{t}=0\}$ and $\varphi_{t}=\inf \left\{s>0:\int_{0}^{s}Z_{u}du>t \right\}$, is a L\'evy process 
with no negative jumps,  stopped at $T_{0}$ and satisfying $\psi(\lambda)=\log \mathrm{I\!E}(e^{-\lambda X_{1}})$, 
with initial position $X_{0}=x$.

Relying on this relationship, Caballero \textit{et al.} \cite[Prop 4]{CLU} provide a pathwise description of the dynamics of a CSBP:
for $(Z_{t}, t\geq 0)$ there exist a standard Brownian motion $B^{Z}$, and an independent Poisson measure $N^{Z}$ on 
$[0,\infty)\times(0,\infty)\times(0,\infty)$ with intensity measure $dt\times d\nu \times \Pi(dr)$ in an enlarged probability space
such that
     \begin{equation}\label{SDECSBP}
      \begin{split}
      Z_{t}=& x + a\int_{0}^{t}Z_{s}ds+\sigma\int_{0}^{t}\sqrt{Z_{s}}dB_{s}^{Z} + \int_{0}^{t}\int_{0}^{Z_{s^{-}}}\int_{1}^{\infty}rN^{Z}(ds, d\nu, dr)\\ 
            &+\int_{0}^{t}\int_{0}^{Z_{s^{-}}}\int_{0}^{1}r\tilde{N}^{Z}(ds, d\nu, dr),
      \end{split}      
     \end{equation}
where $\tilde{N}^{Z}$ is the compensated Poisson measure associated with $N^{Z}$.
   Pathwise properties of stochastic differential equations driven by Brownian motion and Poisson point processes have been studied  
in more  general settings in \cite{DL0}, \cite{FL}  and  \cite{DL}.
  In particular,  strong existence and pathwise uniqueness for \eqref{SDECSBP} is established \cite{FL}.
 Related SDE have also been considered in Bertoin and Le Gall  \cite{BLGII}, \cite{BLGIII}.

\section{CSBPs conditioned to be never extinct as  solutions of SDEs}

\subsection{CSBP conditioned to be never extinct}
  
We assume from now on that $Z$ is a (sub-)critical CSBP such that $\psi(\infty)=\infty$ and 
$\int^{\infty}\frac{d\xi}{\psi(\xi)}<\infty$. Under these and the previous conditions, the process does not explode and there is 
almost surely extinction in finite time. Branching processes conditioned to stay positive were first studied in the continuous-state
framework by Roelly and Rouault \cite{RR}, who proved 
 that for  $Z$  as before, 
       \begin{equation}\label{probcond}
       \mathbb{P}_{x}^{\uparrow}(A):=\lim_{s\uparrow\infty}\mathbb{P}_{x}(A|T>t+s), \quad A\in\sigma(Z_{s}:s\leq t)
      \end{equation}
is a well defined probability measure which satisfies
\[
\mathbb{P}_{x}^{\uparrow}(A)=\mathbb{E}(\textbf{1}_{A}e^{\rho t}\frac{Z_{t}}{x}).
\]  
In particular, $\mathbb{P}_{x}^{\uparrow}(T<\infty)=0$, and  $\{e^{\rho t}Z_{t}: t\geq0\}$ is a martingale under $\mathbb{P}_{x}$. 
Note that $\mathbb{P}_{x}^{\uparrow}$ is the law of the so-called $Q$-\textit{process} (for in-depth looks at this type of processes, 
we refer the reader to \cite{L2}, \cite{S}  and references therein). 
  They also proved that $(Z,\mathbb{P}^{\uparrow})$ has the same law as a CBI with branching mechanism $\psi$ and 
immigration mechanism $\phi(\theta)=\psi'(\theta)-\rho,$ $ \theta\geq0.$ 
This means that $(Z,\mathbb{P}^{\uparrow})$ is a \textit{c\`adl\`ag} $[0,\infty)$-valued process,
and for all  $x,t>0$ and $\theta\geq0$
\[
\mathbb{E}_{x}^{\uparrow}(e^{-\theta Z_{t}})=\exp \{ -xu_{t}(\theta)-\int_{0}^{t}\phi(u_{t-s}(\theta))ds\}, 
\]
where $u_{t}(\theta)$ is the unique solution to \eqref{eqdif}. Note also that $\phi$ is the Laplace exponent of a subordinator. 
        
\subsection{Main Result}

\hspace*{0.5cm} The above result is the key for the study of CSBP conditioned on non-extinction, but we seek a more explicit
description for the paths of $Z$ under $\mathbb{P}^{\uparrow}$.    
To this end, we shall prove that $(Z,\mathbb{P}^{\uparrow})$ has a SDE representation, which agrees with the interpretation
of a CSBP conditioned on non-extinction as a CBI, but also gives us a pathwise description for the conditioned process.
In particular, this result extends Lambert's results for the stable case \cite[Theorem 5.2]{L2} (see below for details) as well 
as equation \eqref{SDECSBP}.

\begin{theorem}\label{teo:principal}
  Under $\mathbb{P}^{\uparrow}$, the process $Z$ is the unique strong solution  of the following stochastic differential equation:
  \begin{equation}\label{SDECSBPCond}     
  \begin{array}{lcl}      
     Z_{t}&=& x+a\int_{0}^{t}Z_{s}ds
              +\sigma\int_{0}^{t}\sqrt{Z_{s}}dB^{\uparrow}_{s} + \int_{0}^{t}\int_{0}^{Z_{s^{-}}}\int_{1}^{\infty}rN^{\uparrow}(ds, d\nu, dr)+\int_{0}^{t}\int_{0}^{Z_{s^{-}}}\int_{0}^{1}r\tilde{N}^{\uparrow}(ds, d\nu, dr)\\
           \\
           &&+\int_{0}^{t}\int_{0}^{\infty}rN^{\star}(ds, dr) +\sigma^{2}t  
   \end{array}
  \end{equation}
where $\{B^{\uparrow}_{t}:t\geq0\}$ is a Brownian motion, $N^{\uparrow}$ and $N^{\star}$ are Poisson measures on 
$[0,\infty)\times (0,\infty)^2$ and $[0,\infty)\times (0,\infty)$  with intensities measures $ds \times d\nu \times \Pi(dr)$
 and $ds \times r\Pi(dr)$, respectively, and these objects are mutually independent 
(as usual, $\tilde{N}^{\uparrow}$ stands for the compensated measure associated 
with $N^{\uparrow}$).
Moreover, the point  processes $N^{\uparrow}$ and $N^{\star}$ can be constructed by change of measure and a marking procedure on an enlargement of the 
probability space where $B^Z$ and $N^Z$ in \eqref{SDECSBP} are defined, which supports and independent i.i.d. sequence of uniform 
random variables in the unit interval.
 \end{theorem}
 This result implies that we can recover $Z$ conditioned on non-extinction as the solution of a SDE driven by a copy of 
$B^{Z}$, a copy of  $N^{Z}$, and a Poisson random measure with intensity $ds \times r\Pi(dr)$, plus a drift.
 (Notice that  taking out the last line, corresponding to a subordinator with drift, one again  obtains equation \eqref{SDECSBP}.) 

\section{Relations to previous results}

\subsection{Stable processes}
  
\hspace*{0.5cm} We will show that,  as  pointed out before,    Lambert's SDE representation of stable branching processes  given in 
 \cite[Theorem 5.2]{L2} can be seen as a special case of Theorem \ref{teo:principal}. 

  Let $X$ be a spectrally positive $\alpha$-stable process with characteristic exponent $\psi$ and characteristic 
measure $\Pi(dr)=kr^{-(\alpha + 1)}dr$, where $k$ is some positive constant and $1<\alpha\leq 2$. 
Let $Z$ be the branching process with branching mechanism $\psi$. 
Thanks to Theorem  \ref{teo:principal} we know that, under $\mathbb{P}^{\uparrow}$, $Z$ satisfies the following stochastic differential 
equation: 
\begin{equation}\label{eq:stableSDE}
  Z_{t}=\int_{0}^{t}\int_{0}^{Z_{s^{-}}}\int_{1}^{\infty}rN^{\uparrow}(ds,d\nu,dr)
        +\int_{0}^{t}\int_{0}^{Z_{s^{-}}}\int_{0}^{1}r\tilde{N}^{\uparrow}(ds,d\nu,dr)
       + \int_{0}^{t}\int_{0}^{\infty}rN^{\star}(ds,dr),
  \end{equation}
where $N^{\uparrow}$ is a Poisson random measure with intensity $ds \times d\nu \times \Pi(dr)$ and $N^{\star}$ is an independent
Poisson random measure with intensity $ds \times r\Pi(dr)$.
  Now, we define 
\[
\theta_{n}= \displaystyle{\frac{r^{\uparrow}_{n}\textbf{1}_{(\nu^{\uparrow}_{n}\leq Z_{t_{n}-})}}{Z_{t_{n}-}^{1/\alpha}}}, 
\]
where $\{(t_{n},r_{n}^{\uparrow}, \nu_{n}^{\uparrow}): n\in \mathbb{N}\}$ are the atoms of $N^{\uparrow}$. 
We claim that, under $\mathbb{P}^{\uparrow}$, $\{(t_{n}, \theta_{n}): n \in \mathbb{N}\}$ are atoms of a Poisson random measure 
$N'$ with intensity $ds \times \Pi(du)$. 
Indeed, for any bounded non-negative predictable process $H$, and any positive bounded function $f$ vanishing at zero,
\[
M_{t}: = \sum_{t_n\leq t} H_{t_n}f(\theta_n) - \int_{0}^{t}H_{s}ds\int_{0}^{\infty}\int_{0}^{\infty} f\left(\frac{r}{Z_{s}^{1/\alpha}}\right)\textbf{1}_{(\nu\leq Z_{s })}d\nu\Pi(dr)
\]
is a martingale. If we change variables, the particular form of $\Pi$ implies that 
\[
M_{t}= \sum_{t_n \leq t} H_{s}f(\theta_n ) - 
         \int_{0}^{t}H_{s}ds\int_{0}^{\infty} f(u) \Pi(du).
\]
  Taking expectations, our claim follows thanks to Lemma \ref{characterization} below.
  Since
$\sum_{t_n \leq t} r^{\uparrow}_{n}\textbf{1}_{(\nu^{\uparrow}_{n}\leq Z_{t_n -})}  
=\sum_{t_n \leq t} Z^{1/\alpha}_{t_n -} \theta_{n},$
we can rewrite \eqref{eq:stableSDE} as
\begin{equation*}
  Z_{t}= \int_{0}^{t}\int_{1}^{\infty}Z^{1/\alpha}_{s-} uN'(ds,du)+\int_{0}^{t}\int_{0}^{1} Z^{1/\alpha}_{s-}u\tilde{N}'(ds,du) + \int_{0}^{t}\int_{0}^{\infty}rN^{\star}(ds,dr).
\end{equation*}
  Defining
\[
X_{t}:=\int_{0}^{t}\int_{1}^{\infty} uN'(ds,du)+\int_{0}^{t}\int_{0}^{1} u\tilde{N}'(ds,du),
\]
by the L\'evy-Ito decomposition it is easy to see that  $X$ is an $\alpha$-stable L\'evy process  with characteristic exponent 
$\psi$. Similarly,  
\[
S_{t}:=\int_{0}^{t}\int_{0}^{\infty}rN^{\star}(ds,dr)
\]
is seen to be an $(\alpha -1)$-stable subordinator. 
Independence of $X$ and $S$ is granted by construction, because the two processes do not have simultaneous jumps. 
Thus, we have $$ dZ_{t}=Z^{1/\alpha}_{t}dX_{t} + dS_{t} ,$$ which corresponds to Lambert's result.
\subsection{CSBP flows as SDE solutions}

 A family of CSBP processes  $Z=\{Z_{t}(a): t\geq 0, a\geq 0\}$  allowing  the initial population size $Z_{0}(a)=a$ to vary,  can be 
constructed simultaneously as a two parameter process  or {\it stochastic flow} satisfying  the branching property.
  This was done by Bertoin and Le-Gall  \cite{BLG0} by using families of  subordinators. In \cite{BLGII}, \cite{BLGIII} they later used 
   Poisson measure driven SDE to  formulate such type of flows in related contexts, including equations  close   to \eqref{SDECSBP}.  
  In the same line, Dawson and Li \cite{DL} proved the existence of strong solutions for stochastic flows of continuous-state branching
processes with immigration, as SDE families driven by white noise processes and Poisson random measures with joint  regularity properties.
   The stochastic equations they study (in particular  equation (1.5) ) are  close to equation \eqref{SDECSBPCond},
 the main difference being  the immigration behavior which in their case only covers linear drifts.
  For simplicity reasons Theorem \ref{teo:principal} is presented  in the case of a Brownian motion and Poisson measure
  driven SDE, but  our  arguments   can be extended to the white-noise and  Poisson measure driven stochastic flow 
considered in \cite{DL} (in absence of immigration).

\section{Proof of the main theorem}
\hspace*{0.5cm}  In \cite{L2}, a suitable marking of Poisson point processes was used to 
firstly construct a stable L\'evy process, conditioned to stay positive, out of the realization of the unconditioned one.
  After time-changing the author takes advantage of the scaling property of $\alpha$-stable processes to derive an SDE for the 
branching process. Our proof is  inspired in his marking argument  but in turn it is carried out directly in the time scale of the CSBP.
We will need the following version of Girsanov's theorem (c.f. Theorem 37 in  Chapter III.8 of  \cite{P}): 
 
   \begin{theorem}
    Let $(\varOmega,\mathcal{F},(\mathcal{F}_{t}),\mathbb{P})$ be a filtered probability space, and  let $M$ be a $\mathbb{P}$-local
martingale with $M_{0}=0$. 
Let $\mathbb{P}^{\star}$ be another probability measure absolutely continuous with respect to $\mathbb{P}$, and let 
$D_{t}=\mathbb{E}(\frac{d\mathbb{P}^{\star}}{d\mathbb{P}}|\mathcal{F}_{t})$. 
Assume that $\langle M,D \rangle$ exists for  $\mathbb{P}$. Then $A_{t}=\int_{0}^{t}\frac{1}{Z_{s^{-}}}d\langle M,Z\rangle_{s}$ 
exists a.s. for the probability $\mathbb{P}^{\star}$, and $M_{t}-A_{t}$ is a $\mathbb{P}^{\star}$-local martingale.
   \end{theorem}
  The following well-known characterization of Poisson point processes will also be useful:
  
\begin{lemma}\label{characterization}
 Let $(\Omega,\mathcal{F}, (\mathcal{F}_{t}),\mathbb{P})$ be a filtered probability space, $(S,\mathcal{S},\eta)$ an arbitrary 
$\sigma$-finite measure space, and $\{(t_{n},\delta_{n}) \in \mathbb{R_{+}}\times S\}$ a countable family of  random variables such 
that $\{ t_{n}\leq t, \delta_{n}\in A \}\in {\cal F}_t$ for all $n \in \mathbb{N}$, $t\geq 0$ and $A\in {\cal S}$, and moreover
\begin{equation}\label{form}
\mathbb{E}\sum\limits_{n:t_{n}\leq t}F_{t_{n}}g(\delta_{n})= \mathbb{E} \int\limits_{0}^{t}F_{s}ds\int\limits_{S}g(x)m(dx) 
\end{equation}
for any nonnegative predictable process $F_s$ and any  nonnegative function $g:S\rightarrow \mathrm{I\!R}$. 
Then, $(t_{n},\delta_{n})_{n\in \mathbb{N}}$ are the atoms of a Poisson random measure $N$ on $\mathbb{R_{+}\times S}$ 
with intensity  $dt \times m(dx)$.
\end{lemma}
\begin{proof} Writing
\[
\displaystyle{e^{\left\{\sum\limits_{t_{n}\leq t}f(\delta_{n}))\right\}}}= 
\sum\limits_{n:t_{n}\leq t}\left[\prod\limits_{k:t_{k}<t_{n}}e^{f(\delta_{k})}\right](e^{f(\delta_{n})}-1)=
\sum\limits_{n:t_{n}\leq t}\left[e^{\sum\limits_{k:t_{k}\leq s}f(\delta_{k})}\right]  (e^{f(\delta_{n})}-1)
\]
we get from  \eqref{form} that
\[
\mathbb{E}\left[e^{\sum\limits_{n:t_{n}\leq t}f(\delta_{n})}\right]=
\int\limits_{0}^{t}\mathbb{E}\left[e^{\sum\limits_{k:t_{k}\leq s}f(\delta_{k})}\right]ds\int_{S}(e^{f(x)}-1)m(dx)
\]
since $F_{s}:=\prod\limits_{ t_{k}<s}e^{f(\delta_{k})}$ is a predictable process. 
  Solving this differential equation yields
\[
\mathbb{E}\left[e^{\sum\limits_{t_{n}\leq t}f(\delta_{n})}\right]=e^{-t\int\limits_{S}(1- e^{f(x)})m(dx)}, 
\]
and the statement follows by  Campbell's formula (see e.g. \cite{Ki})   \end{proof}
\\

 \begin{proof}[Proof of Theorem 3.1]
  We will prove that under the laws $ \mathbb{P}_{x}^{\uparrow}$ the process $Z$  in equation  \eqref{SDECSBP} is a weak solution of \eqref{SDECSBPCond}. Pathwise  
 uniqueness, which then classically implies also  strong existence, can   be  shown as in  \cite{FL}.
  
    We write $B=B^Z$ and $N=N^Z$, and we denote by $\{\mathcal{F}_{t}\}$ the filtration
\[
\mathcal{F}_{t}:=\sigma(B_{s},(r_{n},\nu_{n})\textbf{1}_{t_{n}\leq s};n\in \mathbb{N}, s\leq t), 
\]
where $\{(t_{n},r_{n},\nu_{n})\in [0,\infty)\times(0,\infty)\times(0,\infty)\}_{n\in\mathbb{N}}$ are the atoms of the Poisson point 
process $N$.  We will use the absolute continuity of $\mathbb{P}^{\uparrow}$ w.r.t. $\mathbb{P}$ and  the Radon-Nikodym density 
$D_{t}= \frac{e^{\rho t}Z_{t}}{x}$ applying the previous theorem to the process $\{B_{t}: t\geq 0\}$ and, indirectly,
to the Poisson random measure $N$ and its compensated measure.

 Dealing with the diffusion part is standard since
$d\langle D,B \rangle_{t}=\frac{e^{\rho t}}{x}\sigma\sqrt{Z_{t}}dt,$
so that \[  
B_{t}^{\uparrow}: = B_{t} -\int_{0}^{t}\frac{ d\langle D,B \rangle_{s}}{D_{s}} = B_{t} - \sigma\int_{0}^{t}Z_{s}^{-\frac{1}{2}} ds
\]
is a Brownian motion under $\mathbb{P}^{\uparrow}$ by Girsanov theorem.  

We next study the way the Poisson random measure $N$ is affected by the change 
of probability, which is the main part of the proof. Enlarging the probability space and filtration if needed, we may and shall  assume that  there is a sequence 
$(u_{n})_{n\geq 1}$ of independent  random variables uniformly distributed on $[0,1]$, independent of $B$ and $N$ and such that $u_{n}\textbf{1}_{t_{n}\leq t}$ is 
$\mathcal{F}_{t}$-measurable.
Define random variables $(\Delta_{n},\delta_{n})\in [0,\infty)^2\times [0,\infty)$  by 
\[
  (\Delta_{n},\delta_{n}):= 
  \left\{  
   \begin{array}{lcl}     
   ((0,0),r_{n}\textbf{1}_{(\nu_{n}\leq Z_{t_{n}^{-}})})  & \mbox{ if } & 
   \displaystyle{u_{n} > \frac{D_{t_{n}-}}{D_{t_{n}}} \,  =\frac{Z_{t_{n}-}}{Z_{t_{n}}} }  \mbox{ and } Z_{t_{n}}>0 ,\\
   ((r_{n},\nu_{n}),0)& \mbox{ if } & \displaystyle{u_{n}\leq \frac{D_{t_{n}-}}{D_{t_{n}}}}  \mbox{ and } Z_{t_{n}}>0,\\
   ((0,0),0)   & \mbox{ if } & Z_{t_{n}}=0.\\ 
   \end{array}   
  \right. 
\]                            
  Let $f_{R,\epsilon}$ be a nonnegative function such that for all $(r,\nu,s)$
\begin{itemize}
\item[-] $f_{R,\epsilon}((r,\nu),s)=0$  when $\nu\geq R$, for some fixed $R \geq 0$,
\item[-] $f_{R,\epsilon}((r,\nu),s)=0$  when $r<\epsilon$, for some fixed $0<\epsilon \leq 1$, and
\item[-] $f_{R,\epsilon}((0,0),0)=0$.
\end{itemize}
  For any non-negative predictable process $F$, we have 
\[
\begin{array}{lcl}
 \sum\limits_{t_{n}\leq t}F_{t_{n}}f_{R,\epsilon}(\Delta_{n},\delta_{n})&=&
\sum\limits_{t_{n}\leq t}F_{t_{n}}f_{R,\epsilon}((0,0),r_{n}\textbf{1}_{\{\nu\leq Z_{t_{n}-}\}})
\textbf{1}_{\{ u_{n}>\frac{Z_{t_{n}-}}{Z_{t_{n}}}\}}\\
\\
&&+\sum\limits_{t_{n}\leq t}F_{t_{n}}f_{R,\epsilon}((r_{n},\nu_{n}),0)
\textbf{1}_{\{ u_{n}\leq\frac{Z_{t_{n}-}}{Z_{t_{n}}}\} }. 
\end{array}
\]
  Therefore, since $1- \frac{Z_{t_{n}-}}{Z_{t_{n}}}= \frac{r_{n}\textbf{1}_{\{ \nu_{n}\leq Z_{t_{n}-}\} }}{Z_{t_{n}}}$, the process
\[
\begin{array}{lcl}
 S_{t}&:=&\sum\limits_{t_{n}\leq t}F_{t_{n}}f_{R,\epsilon}(\Delta_{n},\delta_{n}) -\int_{0}^{t}dsF_{s}\int_{0}^{\infty}\int_{0}^{\infty}f_{R,\epsilon}((0,0),r\textbf{1}_{\{\nu\leq Z_{s}\}})
      \displaystyle{\frac{r\textbf{1}_{(\nu\leq Z_{s})}}{Z_{s}+r\textbf{1}_{(\nu\leq Z_{s})}}}\Pi(dr)d\nu\\
      \\   
      && - \int_{0}^{t}dsF_{s}\int_{0}^{\infty}\int_{0}^{\infty} f_{R,\epsilon}((r,\nu),0)
      \displaystyle{\frac{Z_{s}}{Z_{s}+r\textbf{1}_{(\nu\leq Z_{s})}}}\Pi(dr)d\nu 
\end{array}  
\]
is a martingale under $\mathbb{P}$. 
  The quadratic covariation of $S$ and $D$ is given by 
\[
\begin{array}{lcl}
[S,D]_{t} &=& \sum\limits_{t_{n}\leq t}F_{t_{n}}f_{R,\epsilon}(\Delta_{n},\delta_{n})\frac{e^{\rho t_{n}}}{x} r_{n}\textbf{1}_{(\nu_{n}\leq Z_{t_{n}-})}\\
              \\
              &=&\sum\limits_{t_{n}\leq t}F_{t_{n}}f_{R,\epsilon}((0,0),r_{n}\textbf{1}_{\{\nu\leq Z_{t_{n}-}\}})
                 \frac{e^{\rho t_{n}}}{x} r_{n}\textbf{1}_{\{\nu_{n}\leq Z_{t_{n}-}\}}\textbf{1}_{\left\{ u_{n}>\frac{  Z_{t_{n} - }    }{Z_{t_{n}}}\right\} }\\
               \\
               &&+\sum\limits_{t_{n}\leq t}F_{t_{n}}f_{R,\epsilon}((r_{n},\nu_{n}),0)
                \frac{e^{\rho t_{n}}}{x} r_{n}\textbf{1}_{\{\nu_{n}\leq Z_{t_{n}-}\}}\textbf{1}_{  \left\{ u_{n}\leq \frac{  Z_{t_{n} - }    }{Z_{t_{n}}}\right\} }\,, 

\end{array} 
\]
since $S$ is a jump process.
Thus, the conditional quadratic covariation is 
\[
\begin{array}{lcl}
\langle D,S \rangle_{t}&=&\int_{0}^{t}\frac{e^{\rho s}}{x}F_{s}ds\int_{0}^{\infty}\int_{0}^{\infty}f_{R,\epsilon}((0,0),r\textbf{1}_{\{\nu\leq Z_{s}\}})
                           \displaystyle{\frac{r\textbf{1}_{(\nu\leq Z_{s})}}{Z_{s}+r\textbf{1}_{(\nu\leq Z_{s})}}r}d\Pi(dr)d\nu\\
                        \\
                       && + \int_{0}^{t}\frac{e^{\rho s}}{x}F_{s}ds\int_{0}^{\infty}\int_{0}^{\infty}f_{R,\epsilon}((r,\nu),0)
                          \displaystyle{\frac{Z_{s}}{Z_{s}+r\textbf{1}_{(\nu\leq Z_{s})}}r\textbf{1}_{(\nu\leq Z_{s})}}d\Pi(dr)d\nu.
\end{array} 
\]
  Then, using  Girsanov's theorem, we see that the process
\[
\begin{array}{lcl}
S_{t}^{\uparrow}:&=& S_{t} - 
                     \int_{0}^{t}\int_{0}^{\infty}\int_{0}^{\infty}F_{s}f_{R,\epsilon}((0,0),r\textbf{1}_{(\nu\leq Z_{s})})
                     \displaystyle{\frac{r\textbf{1}_{(\nu\leq Z_{s})}}{Z_{s}+r\textbf{1}_{(\nu\leq Z_{s})}}
                            \frac{r}{Z_{s}}}\Pi(dr)d\nu ds\\
                \\
                && - \int_{0}^{t}\int_{0}^{\infty}\int_{0}^{\infty}F_{s}f_{R,\epsilon}((r,\nu),0)
                     \displaystyle{\frac{Z_{s}}{Z_ {s}+r\textbf{1}_{(\nu\leq Z_{s})}}
                     \frac{r\textbf{1}_{(\nu\leq Z_{s})}}{Z_{s}}}\Pi(dr)d\nu ds
\end{array} 
\]    
is a $(\mathcal{F}_{t})$-martingale under $\mathbb{P}^{\uparrow}$.
  By the definition of $S$,
$$
\begin{array}{lcl}
S_{t}^{\uparrow}&=&\sum\limits_{t_{n}\leq t}F_{t_{n}}f_{R,\epsilon}(\Delta_{n},\delta_{n})  - \int_{0}^{t} F_{s}ds\int_{0}^{\infty}\int_{0}^{\infty}\left[f_{R,\epsilon}((0,0),r\textbf{1}_{(\nu\leq Z_{s})})
                     \displaystyle{\frac{r\textbf{1}_{(\nu\leq Z_{s})}}{Z_{s}}} 
                     + f_{R,\epsilon}((r,\nu),0)\right]\Pi(dr)d\nu \\
                \\
                &=& \sum\limits_{t_{n}\leq t}F_{t_{n}}f_{R,\epsilon}(\Delta_{n},\delta_{n})
                - \int_{0}^{t} F_{s}ds\int_{0}^{\infty}\int_{0}^{\infty}\left[f_{R,\epsilon}((0,0),r)
                \displaystyle{\frac{r}{Z_{s}}\textbf{1}_{(\nu\leq Z_{s})}} 
                + f_{R,\epsilon}((r,\nu),0)\right]\Pi(dr)d\nu \\
                \\
                \end{array} 
$$
since $f_{R,\epsilon}((0,0),0)=0$. 
  Recalling that $S^{\uparrow}$ is a $(\mathcal{F}_{t})$-martingale on $\mathbb{P}^{\uparrow}$ starting from $0$, we deduce that
\[
\begin{array}{lcl}
\mathbb{E}^{\uparrow}\left[\sum\limits_{t_{n}\leq t}F_{t_{n}}f_{R,\epsilon}(\Delta_{n},\delta_{n})\right]&=&
   \mathbb{E}^{\uparrow}\left[\int_{0}^{t} F_{s}ds\int_{0}^{\infty}f_{R,\epsilon}((0,0),r)r\Pi(dr)\right]\\
   \\ 
   && + \mathbb{E}^{\uparrow}\left[\int_{0}^{t} F_{s}ds\int_{0}^{\infty}\int_{0}^{\infty}f_{R,\epsilon}((r,\nu),0)\Pi(dr)d\nu\right].
 \end{array} 
\]
  By standard arguments, this formula is also true for any nonnegative function $f$ such that $f((0,0),0)=0$.         
By  Lemma \ref{characterization} we see that, under $\mathbb{P}^{\uparrow}$, $(t_{n},\Delta_{n})_{n\geq0}$ and 
$(t_{n},\delta_{t})_{n\geq 0}$ are atoms of two Poisson point processes $N^{\uparrow}$ and $N^{\star}$  with intensity measures 
$dt \times d\nu \times \Pi(dr)$ and $dt \times r\Pi(dr)$ on $[0,\infty)\times(0,\infty)\times(0,\infty)$ and 
$[0,\infty)\times(0,\infty)$ respectively. 
By construction,  $N^{\uparrow}$ and $N^{\star}$ are independent because they never jump simultaneously.
  Now set
\[
  J_{t}:=\int_{0}^{t}\int_{0}^{Z_{s^{-}}}\int_{1}^{\infty}rN(ds, d\nu, dr)
        =\sum\limits_{t_{n}\leq t} r_{n} \textbf{1}_{(\nu_{n}\leq Z_{t_{n}^{-}})}\textbf{1}_{(r_{n}\geq1)}. 
 \]
  From above, we have
\[
J_{t}= \sum_{t_{n}\leq t} \Delta_{n}^{(1)}\textbf{1}_{(\Delta^{(2)}_{n}\leq Z_{t_{n}^{-}})}\textbf{1}_{(\Delta_{n}^{(1)}\geq1)} 
       +\sum_{t_{n}\leq t}\delta_{n}\textbf{1}_{(\delta_{n}\geq1)}, 
\]
where $\Delta_{n}^{(i)}$ is the $i-$th coordinate of $\Delta_{n}$, $i=1,2$. 
Therefore 
\begin{equation*}
 J(t)=\int_{0}^{t}\int_{0}^{Z_{s^{-}}}\int_{1}^{\infty}rN^{\uparrow}(ds, d\nu, dr)
      +\int_{0}^{t}\int_{1}^{\infty}rN^{\star}(ds,dr).
\end{equation*}
  Finally, given $0<\varepsilon<1$, let $\{\tilde{M}_{t}^{(\varepsilon)},t\geq0\}$ be the $\mathbb{P}$-martingale 
\[
\begin{array}{lcl}
\tilde{M}_{t}^{(\varepsilon)}&:=& \int_{0}^{t}\int_{0}^{Z_{s^{-}}}\int_{\varepsilon}^{1}rN^{Z}(ds,d\nu,dr)           
                    - \int_{0}^{t}\int_{0}^{Z_{s^{-}}}\int_{\varepsilon}^{1}r\, dsd\nu\Pi(dr)\\
\\
                          &=& \sum\limits_{t_{n}\leq t} r_{n} \textbf{1}_{(\nu_{n}\leq Z_{t_{n}^{-}})}\textbf{1}_{(\varepsilon<r_{n}<1)}
               - \int_{0}^{t}\int_{0}^{Z_{s}}\int_{\varepsilon}^{1}r\,dsd\nu\Pi(dr),
\end{array}
\]
which converges in the $L^{2}(\mathbb{P})$ sense when $\varepsilon\rightarrow 0$ to 
$\tilde{M}_{t}:=\int_{0}^{t}\int_{0}^{Z_{s^{-}}}\int_{0}^{1}r\tilde{N}^{Z}(ds,d\nu,dr).$
  In terms of $(\Delta_{n})$ and $(\delta_{n})$, we can write
\[
\begin{array}{lcl}
\tilde{M}^{(\varepsilon)}&=&\left(\sum\limits_{t_{n}\leq t} \Delta^{(1)}_{n} \textbf{1}_{(\Delta^{(2)}_{n}\leq Z_{t_{n}^{-}})}
                            \textbf{1}_{(\varepsilon<\Delta^{(1)}_{n}<1)} 
                             - \int_{0}^{t}\int_{0}^{Z_{s}}\int_{\varepsilon}^{1}rdsd\nu\Pi(dr)\right)\\
                         \\
                         && + \sum\limits_{t_{n}\leq t} \delta_{n}\textbf{1}_{(\varepsilon<\delta_{n}<1)}\\
                         \\
&=&\left(\int_{0}^{t}\int_{0}^{Z_{s^{-}}}\int_{\varepsilon}^{1}rN^{\uparrow}(ds,d\nu,dr) 
                             - \int_{0}^{t}\int_{0}^{Z_{s}}\int_{\varepsilon}^{1}rdsd\nu\Pi(dr)\right)\\
                         \\
                         && + \int_{0}^{t}\int_{\varepsilon}^{1}rN^{\star}(ds,dr).
\end{array}
\]
  Thanks to \cite[Theorem 2.10]{Ky}, the limit as $\varepsilon\to 0$  in the $L^{2}(\mathbb{P}^{\uparrow})$ sense of the $\mathbb{P}^{\uparrow}$-martingale
given by  the first term on the right hand side exists, and  it is equal to the martingale 
$\int_{0}^{t}\int_{0}^{Z_{s^{-}}}\int_{0}^{1}r\tilde{N}^{\uparrow}(ds,d\nu,dr)$,
where $\tilde{N}^{\uparrow}$ is  the compensated measure associated with $N^{\uparrow}$. 
Also, as $\int_{0}^{\infty}(1 \wedge x^{2})\Pi(dx)< \infty$, by \cite[Theorem 2.9]{Ky} the second term on the right hand side 
converges $\mathbb{P}^{\uparrow}$-a.s.,  so we have
\[
\tilde{M}_{t}=\int_{0}^{t}\int_{0}^{Z_{s^{-}}}\int_{0}^{1}r\tilde{N}^{\uparrow}(ds,d\nu,dr) 
              + \int_{0}^{t}\int_{0}^{1}rN^{\star}(ds,dr).
\]

Bringing  all parts together, we have shown  that  $Z$ satisfies under $\mathbb{P}^{\uparrow}$ the desired SDE, 
except for the independence of the processes $B^{\uparrow}$ and $(N^{\uparrow}, N^{\star}$), which we shall establish in what follows.

  Since $N^{\uparrow}$ and $N^{\star}$ have $\sigma$-finite intensities and thanks to the Markov property of the three processes with 
respect to the filtration $(\mathcal{F}_{t})$, it is enough to show that for every $t>s\geq0$, $\zeta \in \mathbb{R}$, 
$\lambda_{k},\gamma_{k} \in \mathbb{R}_{+}$, $k\in\{1,...,m\}$, and $m\in \mathbb{N}$
\[
\begin{array}{lcl} 
\mathbb{E}^{\uparrow}\left[\displaystyle{e^{-\zeta(B^{\uparrow}_{t}-B^{\uparrow}_{s})}
                             e^{-\sum\limits_{k=1}^{m}\lambda_{k}N^{\uparrow}((s,t]\times W_{k})}
                             e^{-\sum\limits_{k=1}^{m}\gamma_{k}N^{\star}((s,t]\times V_{k})}}\bigg{|}\mathcal{F}_{s}\right]\\
\\
  =  \displaystyle{e^{- \frac{\zeta^{2}}{2}(t-s)}e^{\sum\limits_{k=1}^{m}\int_{s}^{t}\int_{W_{k}}(e^{-\lambda_{k}}-1)\Pi(dr)d\nu du}
                   e^{\sum\limits_{k=1}^{m}\int_{s}^{t}\int_{V_{k}}(e^{-\gamma_{k}}-1)r\Pi(dr)du}},     
\end{array} 
\]
where 
$\{W_{k}\}_{k=1}^{m}$ and  $\{V_{k}\}_{k=1}^{m}$ are disjoint sets of $(0,\infty)\times (0,\infty)$ and $(0,\infty)$ such that 
$\int_{0}^{t}\int_{W_{k}}\Pi(dr)d\nu du$ and  $\int_{0}^{t}\int_{V_{K}}r\Pi(dr)du$ are finite.
To that end, set 
$$F(x,y_{1},..,y_{m},z_{1},..,z_{m}):= e^{-\zeta x}e^{-\sum_{k=1}^{m}\lambda_{k}y_{k}}e^{-\sum_{k=1}^{m}\gamma_{k}z_{k}}.$$
Applying It\^{o}'s formula to the semimartingale
\[
X(t)=\left(B^{\uparrow}(t),N^{\uparrow}((0,t]\times W_{1}),..,
      N^{\uparrow}((0,t]\times W_{m}),N^{\star}((0,t]\times V_{1}),..,N^{\star}((0,t]\times V_{m})\right), 
\]
we obtain:
\[
\begin{array}{lcl}
  F(X(t))&=&F(X(s))-\int_{s}^{t}\zeta F(X(u))dB^{\uparrow}_{u} -\sum\limits_{j=1}^{m}\int_{s}^{t}\int_{W_{j}}\lambda_{j}F(X(u))N^{\uparrow}(du, d\nu, dr)\\
         \\
         && -\sum\limits_{j=1}^{m}\int_{s}^{t}\int_{V_{j}}\gamma_{j}F(X(u))N^{\star}(du, dr)+\frac{\zeta^{2}}{2}\int_{s}^{t}F(X(u))du
           + \sum\limits_{s<u\leq t} F(X(u))-F(X(u^{-})) \\
          \\
         &&+\sum\limits_{s<t_{n}\leq t}\sum\limits_{j=1}^{m}\left[\lambda_{j}F(X(t_{n}))\textbf{1}_{\{\Delta_{n}\in W_{j}\}}
           +\gamma_{j}F(X(t_{n}))\textbf{1}_{\{\delta_{n}\in V_{j}\}}\right].
\end{array}
\]
  From above, we deduce that
\[
\begin{array}{lcl}
F(X(t))-F(X(s))&=&\bar{M}_{t}-\bar{M}_{s}+\frac{\zeta^{2}}{2}\int_{s}^{t}F(X(u))du \\
               \\
               && + \sum\limits_{s<t_{n}\leq t}\left[ F\left(X(t_{n}-)+(0,\textbf{1}_{\{\Delta_{n}\in W_{1}\}},..,\textbf{1}_{\{\delta_{n}\in V_{m}\}})\right)-F(X(t_{n}-))\right],              
\end{array}
\]
where $(\bar{M}_{t})$ is a $(\mathcal{F}_{t})$-martingale.
Defining
$
f(\Delta_{n},\delta_{n}):=e^{-\sum\limits_{k=1}^{m}\lambda_{k}\textbf{1}_{\{\Delta_{n}\in W_{k}\}}
                          -\sum\limits_{k=1}^{m}\lambda_{k}\textbf{1}_{\{\delta_{n}\in V_{k}\}}} - 1, $
we have
\begin{equation*}
 F(X(t))-F(X(s))=\bar{M}_{t}-\bar{M}_{s} +\frac{\zeta^{2}}{2}\int_{s}^{t}F(X(u))du
                + \sum\limits_{s<t_{n}\leq t}F(X(t_{n}-))[f(\Delta_{n},\delta_{n})]. 
\end{equation*}
Let now $A \in \mathcal{F}_{s}$. Multiplying both sides by $F(-(X(s)))\textbf{1}_{A}$, yields:
\[
\begin{array}{lcl}
\mathbb{E}^{\uparrow}[F(X(t-s))\textbf{1}_{A}]-\mathbb{P}^{\uparrow}(A)&=&\frac{\zeta^{2}}{2}\int_{s}^{t}\mathbb{E}^{\uparrow}
                                                                \left[F(X(u-s))\textbf{1}_{A}\right]du \\
                     \\
                     && + \int_{s}^{t}\mathbb{E}^{\uparrow}\left[F(X(u-s))\textbf{1}_{A}\right]du
                           \sum\limits_{k=1}^{m}\int_{W_{k}}(e^{-\lambda_{k}}-1)\Pi(dr)d\nu \\
                     \\
                    && + \int_{s}^{t}\mathbb{E}^{\uparrow}\left[F(X(u-s))\textbf{1}_{A}\right]du
                          \sum\limits_{k=1}^{m}\int_{V_{k}}(e^{-\gamma_{k}}-1)r\Pi(dr).
       
    \end{array} 
\]
  Thus, 
\[
\mathbb{E}^{\uparrow}\left[F(X(t-s))\textbf{1}_{A}\right] = \mathbb{P}^{\uparrow}(A) e^{-\frac{\zeta^{2}}{2}(s-t)}
                                                            e^{\sum\limits_{k=1}^{m}\int_{s}^{t}\int_{W_{k}}(e^{-\lambda_{k}}-1)\Pi(dr)d\nu du}
                                                            e^{\sum\limits_{k=1}^{m}\int_{s}^{t}\int_{V_{k}}(e^{-\gamma_{k}}-1)r\Pi(dr)du}
\] 
which means that the three processes are mutually independent, which ends the proof of 
weak existence.

As concerns pathwise uniqueness, we just remark that the proof of Theorem 3.2 in 
 \cite{FL} covers the case of equation \eqref{SDECSBPCond}.   Indeed, if $B^{\uparrow}$, $N^\uparrow$ and $N^{\star}$ are independent processes as before  driving two solutions 
 $\{Z_{t}^{(1)}\}$ and $\{Z_{t}^{(2)}\}$ of \eqref{SDECSBPCond}, setting $\zeta_{t}: = Z_{t}^{(1)} - Z_{t}^{(2)}$
 one gets that 
\begin{equation}\label{eq:z1-z2}
 \begin{array}{lcl}
  \zeta_{t}&=&\zeta_{0} + \int_{0}^{t} a \left(Z_{s}^{(1)} - Z_{s}^{(2)}\right)ds + \int_{0}^{t}\sigma\left(\sqrt{Z_{s}^{(1)}}
               - \sqrt{Z_{s}^{(2)}}\right)dB^{\uparrow}_{s}\\
           \\
          && + \int_{0}^{t}\int_{U_{0}} r \left(\textbf{1}_{(\nu<Z_{s}^{(1)})} - \textbf{1}_{(\nu<Z_{s}^{(2)})}\right)N^{\uparrow}(ds,d\nu,dr)\\
           \\
          && + \int_{0}^{t}\int_{U_{1}} r \left(\textbf{1}_{(\nu<Z_{s}^{(1)})} - \textbf{1}_{(\nu<Z_{s}^{(2)})}\right)\tilde{N}^{\uparrow}(ds,d\nu,dr),       
 \end{array}
\end{equation}
where $U_{0}=[0,\infty)\times[1,\infty)$ and $U_{1}=[0,\infty)\times(0,1)$. From this point on, the proof of  Theorem 3.2 in 
 \cite{FL}  applies, since conditions (2.a,b) and (3.a,b) therein are satisfied. Indeed, in  their notations, we have the intensity measure $\mu(du)=\Pi(dr)d\nu$ for $N^{0}=N^{\uparrow}|_{U_{0}}$ and $N^{1}=N^{\uparrow}|_{U_{1}}$ (where $u=(r,\nu)$), 
 continuous functions on $\mathbb{R}$  given by $b(x):=ax\textbf{1}_{0\leq x} $  and $\sigma(x):= \sigma\sqrt{x}\textbf{1}_{0\leq x}  $, 
and  Borel functions on $\mathbb{R}\times U_{i}$, 
$i=\{0,1\}$ given by $g(x,u)=g_{0}(x,u)=g_{1}(x,u)= r\textbf{1}_{\nu<x}$ such that $g(x,u)+x\geq0$ for $x>0$
 and $g(x,u)=0$ for $x\leq0$. Moreover,  \begin{enumerate}
 \item
  there is a constant $K:= |a| + M \geq 0$ , where $\int_{1}^{\infty} r\Pi(dr)= M<\infty$,
   such that 
  \[
    |ax| + \int_{0}^{\infty}\int_{1}^{\infty} r\textbf{1}_{\nu<x}\Pi(dr)d\nu \leq K(x+1) \,; 
  \]
\item
  there is a non-negative and non-decreasing function $L(x)=(\sigma^{2} + I)(x)$ on  $\mathbb{R}_{+}$, with  $I= \int_{0}^{1} r^{2}\Pi(dr)$, so that
\[
 \sigma^{2}x + \int_{0}^{\infty}\int_{0}^{1} r^{2}\textbf{1}_{\nu<x}\Pi(dr)d\nu \leq L(x) ;
\]
\item 
 there is a continuous non-decreasing function $x\rightarrow b_{2}(x):= x$ on  $\mathbb{R}_{+}$  such that for $b_1(x)= b(x) + b_{2}(x)$, on has  
$$|(a+1)(b_1(x) -b_1(y))| + \int_{0}^{\infty}\int_{1}^{\infty}r\textbf{1}_{y<\nu<x} \Pi(dr)d\nu \leq r(|x - y|) \, ; $$
where $r$ is  the non-decreasing and concave function $r(z)=:(|a+1|+ M)z$ on $\mathbb{R}_{+}$  satisfying $\int_{0_{+}} r(z)^{-1}dz = \infty $; and
\item 
 for every fixed $u \in U_{0}$ the function $x \rightarrow g(x, u)$ is non-decreasing, and  
there is a non-negative and non-decreasing function $\rho(z):=[\sigma^{2} + I]\sqrt{z}$ on $\mathbb{R}_{+}$
 so that $\int_{0_{+}}\rho(z)^{-2}dz= \infty $
and
$$(\sigma \sqrt{x} - \sigma\sqrt{y})^{2} + \int_{0}^{\infty}\int_{0}^{1} r^{2}\textbf{1}_{y<\nu<x} \Pi(dr)d\nu
\leq \rho(|x − y|)^{2}.$$
\end{enumerate}
Conditions 1,2,3 and 4 respectively ensure that hypotheses (2.a,b) and (3.a,b)  in  \cite{FL} hold, and pathwise uniqueness follows.
 \end{proof}

{\bf Acknowledgements}  We would like to thank Julien Berestycki for pointing out  to us relevant references and for several  remarks  that  helped us  to improve earlier versions of this work.

\end{document}